\documentclass{amsart}
\usepackage{array}
\usepackage{ifthen}
\usepackage{url}
\usepackage{amssymb,amsmath, amsfonts}
\newif\iffull

\usepackage{xcolor}
\newtheorem{theorem}{Theorem}[section]
\newtheorem{lemma}[theorem]{Lemma}

\theoremstyle{definition}

\theoremstyle{remark}

\numberwithin{equation}{section}

\def\QQ{\mathbb Q}
\def\ZZ{\mathbb Z}

\newcommand{\klammern}[4][]%
{\ifthenelse{\equal{#1}{}}{\left#2}{\csname#1\endcsname#2}%
#4\ifthenelse{\equal{#1}{}}{\right#3}{\csname#1\endcsname#3}}

\newcommand{\conj}[1]{^{(#1)}}

\newcommand{\abs}[1]{\left\vert#1\right\vert}

\date{\today}

\begin{document}

\title[Repdigits and Consecutive Shifted Tribonacci numbers]{\small Repdigits as Product of Consecutive Shifted Tribonacci Numbers}

\author[Pranabesh Das]{Pranabesh Das}
\address{Department of Mathematics, Xavier University of Louisiana, 
	1 Drexel Drive, New Orleans, Louisiana, 70125 USA}
\email{pdas@xula.edu}

\author{Salah Eddine Rihane}
\address{National Higher School of Mathematics, Scientific and Technology Hub of Sidi Abdellah, P.O. Box 75, Algiers 16093, Algeria.}
\email{salahrihane@hotmail.fr}

\author{Alain Togb\'e}
\address{Department of Mathematics, Statistics, and Computer Science, Purdue University Northwest, 1401 S, U.S. 421, Westville IN 46391, USA}
\email{ atogbe@pnw.edu}




\begin{abstract}
A repdigit is a positive integer that has only one distinct digit in its decimal expansion, i.e., a number has the form $d(10^m-1)/9$ for some $m\geq 1$ and $1 \leq d \leq 9$.  Let $\left(T_n\right)_{n\ge0}$ be the Tribonacci sequence. This paper deals with the presence of repdigits in the products of consecutive shifted Tribonacci numbers.
\end{abstract}

\maketitle

\section{Introduction}\label{sec1}

A positive integer is called a \textit{repdigit} if it has only one
distinct digit in its decimal expansion. The sequence of numbers with
repeated digits is included in Sloane's {\it On-Line Encyclopedia of
	Integer Sequences} (OEIS) \cite{slone} as sequence  A010785. 

On the other hand, the sequence $( T_n)_{n\geq 0}$ of Tribonacci numbers is given by 
$$T_0=0,~T_1=T_2=1,\;\; \mbox{and} \;\; T_{n+3}=T_{n+2}+T_{n+1}+T_n,$$
for all $n\ge 0$. This is the A000073  sequence  in Sloane's {\it On-Line Encyclopedia of Integer Sequences} (OEIS) \cite{slone}.\\

Finding some specific properties of sequences is of big interest since the famous result of Bugeaud, Mignotte, Siksek \cite{BMS}. One can also see \cite{BGL} - \cite{MT}, \cite{RP} - \cite{RT}. Marques and the second author \cite{MT} studied repdigits as products of consecutive Fibonacci numbers. Irmak and the second author \cite{IT} studied repdigits as products of consecutive Lucas numbers. Rayaguru and Panda \cite{RP} studied repdigits as products of consecutive Balancing and Lucas-Balancing numbers. 
Bravo, G\'omez and Luca \cite{BGL} proved that if the product $T_n\cdots T_{n+\ell-1}$ is a repdigit with at least two digits, then $(n, \ell, m, d) = (8, 1, 2, 4)$, i.e. $T_8 = 44$.
It is natural to ask what will happen if we consider consecutive shifted Tribonacci numbers. This is the aim of this paper.\\

Therefore, in this paper, we investigate  repdigits that are written as the product of consecutive shifted Tribonacci numbers. More precisely, we prove the following results.

\begin{theorem}\label{main:thm}
	The Diophantine equation
	\begin{equation}\label{main:eqn}
		(T_{n}+1)\cdots (T_{n+\left(\ell-1\right)}+1)=d\left(\frac{10^{m}-1}{9}\right), 
	\end{equation}
	has no solution in positive integers $n,\ell,m,d,$ with $1\leq d\leq 9$ and $m\geq 2$.
\end{theorem}
\begin{theorem}\label{main:thm2}
	The Diophantine equation
	\begin{equation}\label{main:eqn2}
		(T_{n}-1)\cdots (T_{n+\left(\ell-1\right)}-1)=d\left(\frac{10^{m}-1}{9}\right), 
	\end{equation}
	has no solutions in positive integers $n,\ell,m,d,$ with $1\leq d\leq 9$ and $m\geq 2$.
\end{theorem}
\begin{theorem}\label{main:thm3}
	The Diophantine equation
	\begin{equation}\label{main:eqn3}
		(T_{n}-1)\cdots (T_{n+\left(k-1\right)}-1)(T_{n+k}+1) \cdots (T_{n+k+\left(\ell -1\right)}+1)=d\left(\frac{10^{m}-1}{9}\right),
	\end{equation}
	has no solution in positive integers $n,\ell,m,d,$ with $1\leq d\leq 9$ and $m\geq 2$.
\end{theorem}
\begin{theorem}\label{main:thm4}
	The Diophantine equation
	\begin{equation}\label{main:eqn4}
		(T_{n}+1)\cdots (T_{n+\left(k-1\right)}+1)(T_{n+k}-1) \cdots (T_{n+k+\left(\ell -1\right)}-1)=d\left(\frac{10^{m}-1}{9}\right),
	\end{equation}
	has no solution in positive integers $n,\ell,m,d,$ with $1\leq d\leq 9$ and $m\geq 2$.
\end{theorem}

\noindent One can notice that there is a big difference between our results and those of \cite{BGL} (which of course has its great value). In this paper, we consider shifts to the left, to the right, and a mixture of shifts. Now, we present the outline of this paper. In Section \ref{sec2}, we will recall the results that we will use to prove Theorems \ref{main:thm} - \ref{main:thm4}. In Section \ref{sec3}, first we will use Baker's method and $2$-adic valuation of Padovan numbers to obtain a bound for $n$ that is too high to completely solve equation \eqref{main:eqn}. We will then need to apply twice the reduction method of de Weger to find a very low bound for $n$, which enables us to run a program to find the small solutions of equation \eqref{main:eqn}. We will use the same method in the next sections to prove the remaining theorems.

\section{The Tools} \label{sec2}
\subsection{Properties of the Tribonacci Sequence}
\smallskip
\noindent We start by recalling some useful properties of the Tribonacci sequence. The characteristic equation of $\{T_n\}_{n\geq 0}$ is $z^{3}-z^2-z-1=0$ and has one real root $\alpha$ and two complex roots $\beta$ and $\gamma=\bar{\beta}$. The Binet's formula for the Tribonacci numbers is 
\begin{equation}\label{BT}
	T_{s}=c_{\alpha}\alpha^{s}+c_{\beta}\beta^{s}+c_{\gamma}\gamma^{s}, \quad \text{for all} \quad s\geq 0
\end{equation}
where
\[
c_{\alpha}=\frac{1}{3\alpha^2-2\alpha-1},\quad c_{\beta}=\frac{1}{3\beta^2-2\beta-1}, \quad c_{\gamma}=\frac{1}{3\gamma^2-2\gamma-1}=\overline{c_{\beta}}.
\]
It is easy to see that  $\alpha\in \left(1.83, 1.84\right), |\beta|=|\gamma| \in \left(0.73, 0.74\right), c_{\alpha} \in \left(0.18, 0.19\right)$  and $|c_{\beta}|=|c_{\gamma}| \in \left(0.35, 0.36\right)$.
Since $|\beta|=|\gamma|<1$, setting $e_s:=T_{s}-c_{\alpha}\alpha^{s}$, we have
\begin{equation}\label{error}
	T_{s} = c_{\alpha}\alpha^{s}+e_s, \qquad {\rm with} \quad |e_s|<\frac{1}{\alpha^{s/2}}, \quad \text{for all} ~~ s\geq 1.
\end{equation}
Further, one can easily see that
\begin{equation}\label{E3}
	\alpha^{s-2}\leq T_{s}\leq \alpha^{s-1}, \quad \text{for all} \quad s \geq 1 \quad (\text{see \cite{BGL}}).
\end{equation}
We recall a result of Fac\'{o} and Marques \cite{FM} and Young \cite{Y:2020} on the 2-adic order of shifted Tribonacci number. For a prime number $p$ and a nonzero integer $r$, the $p$-adic order $\upsilon_{p}(r)$ is the exponent of the highest power of a prime $p$ that divides $r$.
\begin{lemma}\label{orderp}
	We have
	$$
	\upsilon_{2}\left(T_{n}+1\right) = \left \{
	\begin{array}{lll}
		0, & \mbox{if }\; n\equiv 0,3 & \pmod{4},\\
		1, & \mbox{if }\; n\equiv 1,2,6 & \pmod{8},\\
		3+\upsilon_{2}(m)+\upsilon_{2}(m-z), & \mbox{if }\; n=8m-3,\\
	\end{array}\right.
	$$
	where $z$ is some 2-adic integer satisfying $z \equiv -601592 \pmod{2^{20}\mathbb{Z}_2}$
	and
	$$
	\upsilon_{2}\left(T_{n}-1\right) = \left \{
	\begin{array}{lll}
		0, & \mbox{if }\; n\equiv 0,3 & \pmod{4},\\
		1, & \mbox{if }\; n\equiv 5 & \pmod{8},\\
		\upsilon_{2}(n+2)-1, & \mbox{if }\; n\equiv 6 & \pmod{8},\\
		\upsilon_{2}(n-2)-1, & \mbox{if }\; n\equiv 2 & \pmod{8},\\
		\upsilon_{2}((n-1)(n+7))-3, & \mbox{if }\; n\equiv 1 & \pmod{8},
	\end{array}\right.
	$$
	for $n\geq 5$. 
\end{lemma}
\subsection{Linear Forms in Logarithms}
The next tools are related to the transcendental approach to solve Diophantine equations. For any non-zero algebraic number $\gamma$ of degree $d$ over $\QQ$, whose minimal polynomial
over $\ZZ$ is $a\prod_{j=1}^w \left(X-\gamma\conj j \right)$, we denote by
\[
h(\gamma) = \frac{1}{w} \left( \log|a| + \sum_{j=1}^w \log\max\left(1,
|{\gamma\conj j}|\right)\right)
\]
the usual absolute logarithmic height of $\gamma$.

To prove our main results, we use lower bounds for linear forms in logarithms to bound the index $n$ appearing in equations \eqref{main:eqn}, \eqref{main:eqn2}, \eqref{main:eqn3} , and \eqref{main:eqn4}. We need the following general lower bound for linear forms in logarithms due to Matveev \cite{M}. See also Theorem 9.4 of \cite{BMS}.

\begin{lemma}\label{lem:Matveev}
	Let $\gamma_1,\ldots ,\gamma_s$ be real algebraic numbers and let  $b_1,\ldots,b_s$ be nonzero rational integer numbers. Let $D$ be the degree of the number field $\QQ(\gamma_1,\ldots ,\gamma_s)$ over $\QQ$ and let $A_j$ be a positive real number satisfying  
	$$
	A_j=\max\{Dh(\gamma),|\log\gamma|,0.16\}\;\text{for}\; j=1,\ldots ,s. 
	$$
	Assume that
	$$
	B\geq\max\{|b_1|,\ldots,|b_s|\}.
	$$
	If $\gamma_1^{b_1}\cdots\gamma_s^{b_s}\neq1$, then
	$$
	|\gamma_1^{b_1}\cdots\gamma_s^{b_s}-1|\geq\exp(-C(s,D)(1+\log B) A_1\cdots A_s),
	$$
	where $C(s,D):=1.4\cdot 30^{s+3}\cdot s^{4.5}\cdot D^2(1+\log D).$
\end{lemma} 
\subsection{Reduction Method}
After getting the upper bound of $n$ which is generally too large, the next step is to reduce it. For this reduction purpose, we present a variant of the reduction method of Baker and Davenport due to de Weger  \cite{Weger:1989}).

Let $\vartheta_1,\vartheta_2,\beta\in \mathbb{R}$ be given, and let $x_1,x_2\in \mathbb{Z}$ be unknowns. Let
\begin{equation}\label{eq4}
	\Lambda =\beta +x_1\vartheta_1 +x_2\vartheta_2.
\end{equation}
Let $c$,  $\delta$ be positive constants. Set $X=\max\{|x_1|,|x_2|\}$. Let $X_0, Y$ be positive. Assume that 
\begin{equation}\label{eq5}
	|\Lambda|<c\cdot \exp(-\delta\cdot Y),
\end{equation}
\begin{equation}\label{eq6}
	Y<X\leq X_0. 
\end{equation}
When $\beta =0$ in \eqref{eq4}, we get 
$$\Lambda =x_1\vartheta_1 +x_2\vartheta_2.$$
Put $\vartheta=-{\vartheta_1}/{\vartheta_2}$. We assume that $x_1$ and $x_2$ are coprime. 
Let the continued fraction expansion of $\vartheta$ be given by $$[a_0,a_1,a_2,\ldots],$$ and let the $k$th convergent of $\vartheta$ be ${p_k}/{q_k}$ for $k=0,1,2,\ldots$.  We may assume without loss of generality that $|\vartheta_1|<|\vartheta_2|$ and that $x_1>0$. We have the following results.

\begin{lemma}{\cite[Lemma 3.2]{Weger:1989}}\label{lem4} 
	Let 
	$$
	A=\max_{0\le k\le Y_0} a_{k+1},
	$$
	where 
	$$
	Y_0=-1+\dfrac{\log(\sqrt{5}X_0+1)}{\log\left(\frac{1+\sqrt{5}}{2}\right)}.
	$$ 
	If \eqref{eq5} and \eqref{eq6} hold for $x_1$, $x_2$ and $\beta=0$, then 
	\begin{equation}\label{eq10}
		Y<\frac{1}{\delta}\log \left(\frac{c(A+2)X_0}{|\vartheta_2|}\right).
	\end{equation}
\end{lemma} 

When $\beta\neq 0$ in \eqref{eq4}, put $\vartheta=-{\vartheta_1}/{\vartheta_2}$ and $\psi={\beta}/{\vartheta_2}$. Then we have 
$$\frac{\Lambda}{\vartheta_2}=\psi-x_1\vartheta+x_2.$$ Let ${p}/{q}$ be a convergent of $\vartheta$ with $q>X_0$. For a real number, $x$ we let $\| x\|=\min\{|x-n|, n\in {\mathbb Z}\}$ be the distance from $x$
to the nearest integer. We have the following result.
\begin{lemma}{\cite[Lemma 3.3]{Weger:1989}}\label{lem5} 
	Suppose that 
	$$\parallel q \psi\parallel>\frac{ 2X_0}{q}.$$ 
	Then, the solutions of \eqref{eq5} and \eqref{eq6} satisfy 
	$$
	Y<\frac{1}{\delta}\log\left(\frac{q^2c}{|\vartheta_2|X_0} \right).$$
\end{lemma}

We conclude this section by recalling the following lemma that we need in the sequel: 
\begin{lemma}{\cite[Lemma 2.2]{Weger:1989}}\label{lem:Weger}
	Let $a, x \in \mathbb{R}$ and $0 < a < 1$. If $|x| < a$, then
	$$
	\abs{\log (1+x)} < \dfrac{-\log (1-a)}{a} |x|
	$$
	and
	$$
	|x| < \dfrac{a}{1-e^{-a}} \abs{e^x-1}.
	$$
\end{lemma}

\section{Proof of Theorem \ref{main:thm}}\label{sec3}
\begin{proof}
	In this section, we will prove our first main result in two steps: find some bounds on the variables using Baker's method and then reduce these bounds by the reduction method.
	\subsection{Absolute bounds on the variables}
	We start by giving the number of factors $\ell$ in the Diophantine equation \eqref{main:eqn}.
	\begin{lemma}\label{lem:maj-l}
		If the Diophantine equation \eqref{main:eqn} has solutions, then $\ell \leq 7$.
	\end{lemma}
	\begin{proof}
		Note that for all $1\leq d\leq 9$, we have  
		$$
		\upsilon_{2}\left(d\left(\frac{10^{m}-1}{9}\right)\right) = \upsilon_{2}(d) \leq 3.
		$$ 
		So, if $\upsilon_2 ((T_{n}+1)(T_{n+1}+1)\cdots (T_{n+(\ell-1)}+1))\geq 4$, then the Diophantine equation \eqref{main:eqn} has no solution.
		
		Let $x\in \left\{0,1,2,\ldots,8\right\}$ such that $n\equiv x\pmod{8}$.  Suppose that $x=6$, hence $n \equiv 6 \pmod 8$, $n+3 \equiv 1 \pmod 8$, $n+4 \equiv 2 \pmod 8$ and $n+7 \equiv 5 \pmod {8}$.  So, by Lemma~\ref{orderp} we get
		$$
		\upsilon_2((T_n+1)(T_{n+1}+1)\cdots (T_{n+7}+1))\geq 6.
		$$
		Therefore, the Diophantine equation \eqref{main:eqn} has no solution if $\ell\geq 8$ in this case.
		
		The other cases can be treated by using a similar method. As a conclusion, 
		we get Table \ref{Table1}. Thus, we deduce that $\ell\leq 7$. 
		
		\begin{table}[!h]
			\centering
			\begin{tabular}[t]{|l|l|l|}
				\hline
				$\ell$ & $x$ & $\upsilon_{2}((T_{n}+1)(T_{n+1}+1)\cdots (T_{n+(\ell-1)}+1))$ \\
				\hline
				2 & 5 & $\geq 4 $\\
				\hline
				3 & 4 & $\geq 4$\\
				\hline
				4 & 2, 3  & $\geq 4$, $\geq 4$ \\
				\hline
				5 & 1 & $\geq5 $\\
				\hline
				6 & 0 & $\geq 5$\\
				\hline
				7 & 7  & $\geq 7$, $\geq 5$\\
				\hline
				8 & 6 & $\geq 6$\\
				\hline
			\end{tabular}
			\caption{2-adic order of product of consecutive shifted Tribonacci numbers on the form $(T_i+1)$}\label{Table1}
		\end{table}
		This completes the proof of Lemma \ref{lem:maj-l}. 
	\end{proof}
	
	Now, we give an upper bound for $n$ and $m$.
	\begin{lemma}\label{redn}
		If $(n,\ell,m,d)$ is a positive integer solution of the Diophantine equation \eqref{main:eqn} with $n\geq 15$, $m\geq 2$,  $1\leq d\leq 9$ and $1\leq \ell \leq 7$, then
		\[
		m\leq\ell n +\ell(\ell+1)/2\quad \text{and} \quad  n< 2.4\times 10^{16}.
		\]
	\end{lemma}
	\begin{proof}
		By \eqref{E3}, for all $u\geq 1$ we have
		\begin{equation}\label{E3-2}
			T_{u}+1<\alpha^{u-1}+1<2\alpha^{u-1}<\alpha^{u+1}.
		\end{equation}
		Thus, by \eqref{main:eqn} we get
		\[
		10^{m-1}<d\left(\dfrac{10^m-1}{9}\right)=(T_n+1)(T_{n+1}+1)\cdots (T_{n+(\ell -1)}+1)< \alpha^{\ell n +\frac{\ell(\ell+1)}{2}}.
		\]
		Thus, we get
		\begin{equation}
			\label{cotam}
			m\leq \ell n +\ell(\ell+1)/2.
		\end{equation}
		Now, by \eqref{error}, we obtain that
		\begin{eqnarray}\label{product-Padovan}
			& & (T_{n}+1) \cdots (T_{n+(\ell-1)}+1) \nonumber \\ 
			&=&(c_{\alpha}\alpha^{n}+e_n+1)\cdots(c_{\alpha}\alpha^{n+(\ell-1)}+e_{n+\ell-1}+1)\\ \nonumber
			&=& c_{\alpha}^{\ell}\alpha^{\ell n+\ell(\ell-1)/2} + r_1(c_{\alpha},\alpha,n,\ell)\nonumber
		\end{eqnarray}
		where $r_1(c_{\alpha},\alpha,n,\ell)$ involves the part of the expansion of the previous line that contains the product of powers of $c_{\alpha},~\alpha$ and the errors $e_i$, for $i=n,\ldots n+(\ell-1)$.
		Moreover, $r_1(c_{\alpha},\alpha,n,\ell)$ is the sum of $2186$ terms with maximum absolute value $c_{\alpha}^{\ell-1}\alpha^{(\ell-1)n +\ell(\ell-1)/2}\alpha^{-n/2}$.
		
		The equality \eqref{product-Padovan} enables us to express \eqref{main:eqn} as
		\[
		\frac{d}{9}10^{m}-c_{\alpha}^{\ell}\alpha^{\ell n+\ell(\ell-1)/2}=\frac{d}{9}+r_1(c_{\alpha},\alpha,n,\ell).
		\]
		Dividing by $c_{\alpha}^{\ell}\alpha^{\ell n+\ell(\ell-1)/2}$ and taking the absolute value, we deduce that
		\begin{align}
			\label{Gamma1}
			\left|\Gamma_1\right|&\leq \left(\frac{d}{9}+|r_1(c_{\alpha},\alpha,n,\ell)|\right)\cdot c_{\alpha}^{-\ell}\alpha^{-(\ell n+\ell(\ell-1)/2)}\nonumber\\
			&< (1+2186c_{\alpha}^{\ell-1}\alpha^{(\ell-1)n+\ell(\ell-1)/2}\alpha^{-n/2})\cdot c_{\alpha}^{-\ell}\alpha^{-(\ell n+\ell(\ell-1)/2)}\\
			&\leq 2187c_{\alpha}^{-1}\alpha^{-3n/2}<11964\alpha^{-3n/2},\nonumber
		\end{align}
		where 
		\begin{equation}
			\Gamma_1=\frac{d}{9c_{\alpha}^{\ell}}\alpha^{-(\ell n+\ell(\ell-1)/2)}10^{m}-1.
		\end{equation}
		To find a lower bound for $\Gamma_1$, we take the parameters $s:=3,$
		\[
		(\eta_{1}, b_1):=((d/9)c_{\alpha}^{-\ell},1), ~~ (\eta_{2},b_2):=(\alpha,-(\ell n+\ell(\ell-1)/2))\quad {\rm and} \quad(\eta_{3},b_3):=(10,m),
		\]
		in Theorem \ref{lem:Matveev}. For our choices, we have $\mathbb{L}:=\mathbb{Q}(\alpha)$, which has $d_{\mathbb{L}}:=3$. To apply Theorem \ref{lem:Matveev} , it is necessary to show that $\Gamma\neq 0$. If we assume the contrary, we get
		\[
		d\cdot10^{m}/9=c_{\alpha}^{\ell}\alpha^{\ell n+\ell(\ell-1)/2}.
		\]
		Conjugating the above relation by the automorphism $\sigma :=( \alpha \beta)$ and then taking absolute values on both sides of the resulting equality, we obtain
		\[
		1<d\cdot10^{m}/9=|c_{\beta}|^{\ell}|\beta|^{\ell n+\ell(\ell-1)/2}<1.
		\]
		This is a contradiction. Thus, $\Gamma_1\neq 0$. Next, we give estimates to $A_i$ for $1,2,3$. $h(\eta_{1})\leq h(d)+h(9)+\ell h(c_{\alpha})\leq 2\log 9+\ell h(c_{\alpha})$, $h(\eta_{2})=\frac{1}{3}\log \alpha$ and $h(\eta_{3})=\log 10$. Now we need to estimate $h(c_{\alpha})$. For it, the minimal polynomial of $c_{\alpha}$ is $44X^{3}-4X-1$. So, $h(c_{\alpha})=\frac{1}{3}\log 44$ and $h(\eta_{1})\leq 2\log 9+\frac{7}{3}\log 44$. Thus, we take $A_{1}:=40$, $A_{2}:=0.7$ and $A_{3}:=7$. Finally, by \eqref{cotam} and the fact that $\ell\leq 7$ we take $B:=8 n +28$. Applying Matveev's theorem, we get a lower bound for $|\Gamma_1|$, which, by comparing it to \eqref{Gamma1}, leads to
		$$
		\frac{3n}{2}\log \alpha -\log 11964<5.31\times 10^{14}(1+\log(7n+28)).\\
		$$
		Hence,
		\[
		n<5.81\times 10^{14}(1+\log (7n+28)).
		\]
		Therefore, by using Maple, we obtain $n<2.4\times 10^{16}$. 
	\end{proof}
	
	\subsection{Reducing $ n$}
	To lower the bound of $n$, we will use  Lemma \ref{lem5}.
	
	Let 
	$$
	\Lambda_1 :=m\log 10-(\ell n+\ell(\ell-1)/2)\log \alpha + \log (d/9c_{\alpha}^{\ell}).$$
	Therefore, \eqref{Gamma1} can be rewritten as $|e^{\Lambda_1}-1| < 11964\alpha^{-3n/2}$. Further, if $n\geq 15$ then $|\Gamma_1|<0.02$. So by applying Lemma \ref{lem:Weger}, we deduce that
	$$
	|\Lambda_1|<-\dfrac{\log(1-0.02)}{0.02} |\Gamma_1|< 12086 \exp\left( -0.6 n\right).
	$$
	Taking
	$$
	\vartheta_1:=-\log\alpha,\quad \vartheta_2:=\log 10,\quad \psi:= \log \left(\dfrac{d}{9c_{\alpha}}\right),\quad c:=12086,\quad \delta:=0.6$$
	in Lemma \ref{lem5}. Further, as $\max\{m,\ell n+\ell (\ell -1)/2\}<1.7\times 10^{17}$, then we  take $X_0=1.7\times 10^{17}$.
	Using \texttt{Maple}, we find that
	\[
	q_{42}=152414933276058910307
	\]
	satisfies  the conditions of Lemma \ref{lem5} for all $1\leq a\leq 9$ and $1\leq\ell\leq 7$. Therefore, by Lemma \ref{lem5} Diophantine equation \eqref{main:eqn} has solutions then 
	$$
	n\leq \dfrac{1}{0.6} \times \log\left(\dfrac{152414933276058910307^2 \times 12086}{\log 10 \times 1.7\times 10^{17}}\right)< 104.
	$$ 
	Now, we will proceed a second reduction of the bound of $n$. In this application of Lemma \ref{lem5}, we take $X_0=749$ and find that $q_{12}=686323$ verifies the conditions of Lemma \ref{lem5}. Thus, we obtain
	$$
	n\leq \dfrac{1}{0.6} \times \log\left(\dfrac{686323^2 \times 12086}{\log 10 \times 749}\right)<49.
	$$ 
	Hence, it remains to check equation \eqref{main:eqn}, for $1\leq n \leq 48$, $1\leq \ell \leq 7$, $2\leq m \leq 364$ and $1\leq d \leq 9$.  A quick inspection using Maple reveals that the Diophantine equation \eqref{main:eqn} has no solutions. This completes the proof of Theorem \ref{main:thm}.
\end{proof}

\section{Proof of Theorem \ref{main:thm2}}
\begin{proof}
	
	In this section, we will prove our second main result using the same method as for the proof of Theorem \ref{main:thm}.
	\subsection{Absolute Bounds on the Variables}
	First of all, we give the number of factors in the Diophantine equation \eqref{main:eqn2}. 
	\begin{lemma}\label{lem:maj-k}
		If the Diophantine equation \eqref{main:eqn2} has a solution  then $\ell\leq 6$
	\end{lemma}
	\begin{proof}
		Let $x\in \left\{0,1,2,\ldots,7\right\}$ such that $n\equiv x\pmod{8}$.  Assume that $x=1$, hence $n=8k+1$.  So by Lemma~\ref{orderp} we obtain
		$$
		\upsilon_2(T_n-1)= \upsilon_2((n-1)(n+7))-3=\upsilon_2(8k(8k+8))\geq 4.
		$$
		We proceed as in this case to show the other results in Table \ref{Table2}.
		
		\begin{table}[!h]
			\centering
			\begin{tabular}[t]{|l|l|l|}
				\hline
				$\ell$ & $x$ & $\upsilon_{2}((T_{n}-1)(T_{n+1}-1)\cdots (T_{n+(\ell-1)}-1))$ \\
				\hline
				1 & 1 & $\geq 4$\\
				\hline
				2 & 0 & $\geq 4$\\
				\hline
				4 & 6, 7 & $\geq 5$, $\geq 5$\\
				\hline
				5 & 2, 5 & $\geq 6$, $\geq 6$\\
				\hline
				6 & 4 & $\geq 6$\\
				\hline
				7 & 3 & $\geq 6$\\
				\hline
			\end{tabular}
			\caption{2-adic order of product of consecutive shifted Tribonacci numbers on the form $T_i-1$}
			\label{Table2}
		\end{table}
		
		As $\upsilon_{2}\left(d\left(\frac{10^{m}-1}{9}\right)\right) = \upsilon_{2}(d) \leq 3$ for all $1\leq d\leq 9$, it then follows from Table \ref{Table2} that $\ell\leq 6$.
	\end{proof}
	Now, we will show the following lemma.
	\begin{lemma}\label{redn-2}
		If $(n,\ell,m,d)$ is a positive integer solution of \eqref{main:eqn2} with $n\geq 20$, $m\geq 2$, $1\leq d\leq 9$ and $1\leq \ell \leq 6$ then
		\[
		m\leq \ell n+\ell (\ell-3)/2\quad \text{and}\quad	n < 2.4\times 10^{16}.
		\]
	\end{lemma}
	\begin{proof}
		First, assume that $n\geq 20$. Combining \eqref{main:eqn2} and \eqref{E3}, we obtain
		\[
		10^{m-1}<\dfrac{d(10^m-1)}{9}=(T_n-1)(T_{n+1}-1)\cdots (T_{n+(\ell-1)}-1)<\alpha^{\ell n+\frac{\ell(\ell-3)}{2}}.
		\]
		Thus,
		\begin{equation}
			\label{cotam-2}
			m\leq\ell n +\ell(\ell-3)/2.
		\end{equation}
		Now, by \eqref{error}, we get that
		\begin{eqnarray}\label{product-Perrin}
			(T_{n}-1) \cdots (T_{n+(\ell-1)}-1)&=&(c_{\alpha}\alpha^{n}+e_n-1)\cdots(c_{\alpha}\alpha^{n+(\ell-1)}+e_{n+\ell-1}-1)\\
			&=&c_{\alpha}^{\ell}\alpha^{\ell n+\ell(\ell-1)/2} + r_2(\alpha,n,\ell)\nonumber
		\end{eqnarray}
		where $r_2(c_{\alpha},\alpha,n,\ell)$ involves the part of the expansion of the previous line that contains the product of powers of $~\alpha$ and the errors $e_i$, for $i=n,\ldots n+(\ell-1)$. Moreover, $r_2(c_{\alpha},\alpha,n,\ell)$ is the sum of $728$ terms with maximum absolute value $c_{\alpha}^{\ell-1}\alpha^{(\ell-1)n +\ell(\ell-1)/2}\alpha^{-n/2}$.
		
		Using equation \eqref{product-Perrin}, we express equation \eqref{main:eqn2} into the form
		\[
		\frac{d}{9}10^{m}-c_{\alpha}^{\ell}\alpha^{\ell n+\ell(\ell-1)/2}=\frac{d}{9}+r_2(\alpha,n,\ell).
		\]
		Multiplying both sides by $c_{\alpha}^{-\ell}\alpha^{-(\ell n+\ell(\ell-1)/2)}$ and taking the absolute value, we see that
		\begin{align}
			\label{cota1-2}
			\left|\Gamma_2\right|&\leq \left(\frac{d}{9}+|r_2(c_{\alpha},\alpha,n,\ell)|\right)\cdot c_{\alpha}^{-\ell}\alpha^{-(\ell n+\ell(\ell-1)/2)}\\
			&< (1+728c_{\alpha}^{\ell-1}\alpha^{(\ell-1)n+\ell(\ell-1)/2}\alpha^{-n/2})\cdot c_{\alpha}^{-\ell}\cdot \alpha^{-(\ell n+\ell(\ell-1)/2)}\nonumber\\
			&\leq729c_{\alpha}^{-1}\alpha^{-3n/2}<3988\alpha^{-3n/2},\nonumber
		\end{align}
		where 
		\begin{equation}
			\Gamma_2=\frac{d}{9c_{\alpha}}\alpha^{-(\ell n+\ell(\ell-1)/2)}10^{m}-1.
		\end{equation}
		Now, we will use Matveev's theorem to find a lower bound for $\Gamma_2$, with the parameters
		$$
		s:=3, ~~ (\eta_{1}, b_1):=((a/9c_{\alpha}),1), 
		$$
		$$ (\eta_{2},b_2):=(\alpha,-(\ell n+\ell(\ell-1)/2))\quad {\rm and} \quad(\eta_{3},b_3):=(10,m).
		$$
		The number field containing $\eta_{1},\eta_{2},\eta_{3}$ is $\mathbb{L}:=\mathbb{Q}(\alpha)$, which has $D:=3$. 
		As above, one can justify that $\Gamma_2\neq 0$. Also, we can take $A_{1}:=40$, $A_{2}:=0.7$ and $A_{3}:=7$. According to \eqref{cotam-2} and the fact that $\ell\leq 6$, we take $B:=6 n +15$.  Applying Matveev's theorem we get a lower bound for $|\Gamma_2|$ and taking into account inequality \eqref{cota1-2}, we obtain
		\[
		\exp\left(-5.31\times 10^{14}(1+\log (6n+15)\right)<\frac{3988}{\alpha^{3n/2}}.
		\]
		Taking the logarithm of both sides of the above inequality, we get
		$$
		\frac{3n}{2}\log \alpha -\log 3988<5.31\times 10^{14}(1+\log(6n+15)).\\
		$$
		Hence,
		\[
		n<5.81\times 10^{14}(1+\log (6n+15)).
		\]
		Solving the above inequality gives $n<2.4\times 10^{16}$.
	\end{proof}
	
	\subsection{Reducing $n$}
	To lower the bound of $n$, we will use  Lemmas  \ref{lem4} and \ref{lem5}, i.e. we will reduce this bound twice.
	
	Let 
	$$
	\Lambda_2 :=m\log 10-(\ell n+\ell(\ell-1)/2)\log \alpha + \log (d/9c_{\alpha}).
	$$
	Therefore, inequality \eqref{cota1-2} can be rewritten as $|e^{\Lambda_2}-1| < 3988\alpha^{-3n/2}$. Furthermore, if $n\geq 15$ then $|\Gamma_1|<0.005$. We apply Lemma \ref{lem:Weger} to obtain
	$$
	|\Lambda_2|<-\dfrac{\log(1-0.005)}{0.005} |\Gamma_2|< 3999 \exp\left( -0.6 n\right).
	$$
	Put
	$$
	\vartheta_1:=-\log\alpha,\quad \vartheta_2:=\log 10,\quad \psi:= \log \left(\dfrac{d}{9c_{\alpha}}\right),\quad c:=12086,\quad \delta:=0.6.$$
	Moreover, as $\max\{m,\ell n+\ell (\ell -1)/2\}<1.5\times 10^{17}$, then we take $X_0=1.5\times 10^{17}$.
	With \texttt{Maple}, we check that
	\[
	q_{42}=152414933276058910307
	\]
	satisfies  the conditions of Lemma \ref{lem5} for all $1\leq a\leq 9$ and $1\leq\ell\leq 6$. Therefore, Lemma \ref{lem5} tells us that if Diophantine equation \eqref{main:eqn2} has solutions then 
	$$
	n\leq \dfrac{1}{0.6} \times \log\left(\dfrac{152414933276058910307^2 \times 3999}{\log 10 \times 1.5\times 10^{17}}\right)< 102.
	$$ 
	Now, we will reduce again the new bound of $n$. We take $X_0=621$ and see that $q_{12}=686323$ satisfies the conditions of Lemma \ref{lem5}. Thus, we obtain
	$$
	n\leq \dfrac{1}{0.6} \times \log\left(\dfrac{686323^2 \times 3999}{\log 10 \times 621}\right)<47.
	$$ Hence, it remains to check equation \eqref{main:eqn} for $1\leq n \leq 46$, $1\leq \ell \leq 6$, $2\leq m \leq 291$ and $1\leq d \leq 9$.  By a fast computation with  Maple in these ranges, we conclude that Diophantine equation \eqref{main:eqn2} has no solutions. This completes the proof of Theorem \ref{main:thm2}.
\end{proof}

\section{Proofs of Theorems \ref{main:thm3} and \ref{main:thm4} }
\begin{proof}
	Using the above method, we will show only Theorem \ref{main:thm3} because the proof of Theorem \ref{main:thm4} is similar. 
	\subsection{Absolute Bounds on the Variables}
	As seen before we claim that $k\leq 6$ and $\ell\leq 7$. Otherwise Diophantine equation \eqref{main:eqn3} has no solutions. This is a direct result from Lemmas \ref{lem:maj-l} and \ref{lem:maj-k}.
	
	Now, we will prove the following lemma.
	\begin{lemma}\label{redn-3}
		If $(n,\ell,k,m,d)$ is a positive integer solution of \eqref{main:eqn3} with $n\geq 30$, $m\geq 2$, $1\leq d\leq 9$, $1\leq k\leq 6$ and $1\leq \ell \leq 7$, then
		\[
		m\leq(k+\ell) n+\frac{k(k-3)}{2}+\frac{\ell(2k+\ell+1)}{2} \quad \text{and}\quad	n < 3.8\times 10^{16}.
		\]
	\end{lemma}
	\begin{proof}
		First, assume that $n\geq 30$. Combining \eqref{main:eqn3}, \eqref{E3} and \eqref{E3-2}, we get
		\[
		\begin{array}{lcl}
			&& 10^{m-1}<\dfrac{d(10^m-1)}{9}\\
			&=&(T_n+1)(T_{n+1}+1)\cdots (T_{n+(k-1)}+1)(T_{n+k}-1)\cdots (T_{n+k+\left(\ell -1\right)}-1)\\
			&<&\alpha^{(k+\ell) n+\frac{k(k-3)}{2}+\frac{\ell(2k+\ell+1)}{2}}.
		\end{array}
		\]
		Thus, we get
		\begin{equation}
			\label{cotam-3}
			m\leq (k+\ell) n+\frac{k(k-3)}{2}+\frac{\ell(2k+\ell+1)}{2}.
		\end{equation}
		Now, by \eqref{error}, we get that
		\begin{eqnarray}\label{product-Perrin-Padovan}
			& & (T_n-1)\cdots (T_{n+(k-1)}-1)(T_{n+k}+1)\cdots (T_{n+k+\left(\ell -1\right)}+1)\\
			&=&                           \alpha^{(k+\ell) n+k(k-1)/2+\ell(2k\ell-1)/2} + r_3(c_{\alpha},\alpha,n,\ell),\nonumber
		\end{eqnarray}
		where $r_3(c_{\alpha},\alpha,n,\ell,k)$ involves the part of the expansion of the previous line that contains the product of powers of $c_{\alpha}~\alpha$ and the errors $e_{i}$ for $i=n,\ldots n+k,\ldots, n+k+(\ell-1)$.
		Moreover, $r_3(c_{\alpha},\alpha,n,\ell)$ is the sum of $1594322$ terms with maximum absolute value $c_{\alpha}^{k+\ell-1}\alpha^{(k+\ell-1) n+k(k-1)/2+\ell(2k+\ell-1)/2}\alpha^{-n/2}$.
		
		Using equation \eqref{product-Perrin-Padovan},  we write \eqref{main:eqn3} into the form
		\[
		\frac{d}{9}10^{m}-c_{\alpha}^{\ell+k}\alpha^{(k+\ell) n+k(k-1)/2+\ell(2k+\ell-1)/2}=\frac{d}{9}+r_3(c_{\alpha},\alpha,n,\ell).
		\]
		Multiplying through by $c_{\alpha}^{-k-\ell}\alpha^{-((k+\ell) n+k(k-1)/2+\ell(2k+\ell-1)/2)}$ and taking the absolute value, we deduce that
		\begin{align}
			\label{cota1-3}
			\left|\Gamma_3\right|&\leq  \left(\frac{d}{9}+|r_3(c_{\alpha},\alpha,n,\ell)|\right)\cdot \alpha^{-((k+\ell) n+k(k-1)/2+\ell(2k+\ell-1)/2)}\nonumber\\
			&< (1+1594322c_{\alpha}^{k+\ell-1}\alpha^{(k+\ell-1) n+k(k-1)/2+\ell(2k+\ell-1)/2}\alpha^{-n/2})\\
			&\cdot c_{\alpha}^{-k-\ell} \alpha^{-((k+\ell) n+k(k-1)/2+\ell(2k+\ell-1)/2)}\nonumber\\
			&\leq1594323c_{\alpha}^{-1}\alpha^{-3n/2}<8721506\alpha^{-3n/2},\nonumber
		\end{align}
		where 
		\begin{equation}
			\Gamma_3=\frac{d}{9c_{\alpha}^{k+\ell}}\alpha^{(k+\ell) n+k(k-1)/2+\ell(2k+\ell-1)/2}10^{m}-1.
		\end{equation}
		Now, we will apply Matveev's theorem to find a lower bound for $\Gamma_3$.
		Put
		$$
		s:=3,\;\; (\eta_{1}, b_1):=((d/9c_{\alpha}^{k+\ell}),1),
		$$
		$$ (\eta_{2},b_2):=(\alpha,-((k+\ell) n+k(k-1)/2+\ell(2k\ell-1)/2)), ~~ (\eta_{3},b_3):=(10,m).
		$$
		The number field containing $\eta_{1},\eta_{2},\eta_{3}$ is $\mathbb{L}:=\mathbb{Q}(\alpha)$, which has $D:=3$. 
		As before, one can prove that $\Gamma_3=\eta_{1}^{b_{1}}\eta_{2}^{b_{2}}\eta_{3}^{b_{3}}-1 \neq 0$. 
		Next, $h(\eta_{1})\leq h(d)+h(9)+(k+\ell)h(c_{\alpha})\leq 2\log 9+\frac{13\log44}{3}$, $h(\eta_{2})=\frac{1}{3}\log \alpha$ and $h(\eta_{3})=\log 10$. Thus, we can take $A_{1}:=62.4$, $A_{2}:=0.7$ and $A_{3}:=7$. According to \eqref{cotam-3} and the facts $k\leq 6$ and $\ell \leq 7$, we take $B:=13n+85$. Applying Matveev's theorem, we get a lower bound for $|\Gamma_3|$, which in comparison to \eqref{cota1-3}, gives
		\[
		\exp\left(-8.27\times 10^{14}(1+\log (13n+85)\right)<\frac{8721506}{\alpha^{3n/2}}.
		\]
		We take the logarithm of both sides of the above inequality to obtain
		$$
		\frac{3n}{2}\log \alpha -\log 8721506<8.27\times 10^{14}(1+\log(13n+85)).\\
		$$
		Hence, we deduce that
		\[
		n<9.05\times 10^{14}(1+\log (13n+85)).
		\]
		Therefore, we obtain $n<3.8\times 10^{16}$.
	\end{proof}
	
	\subsection{Reducing $n$}
	To lower the bound of $n$, we will use twice Lemma \ref{lem5} to reduce the above bound of $n$.
	
	Let 
	$$
	\Lambda_3 :=m\log 10-((k+\ell) n+k(k-1)/2+\ell(2k+\ell-1)/2)\log \alpha + \log (d/9c_{\alpha}^{\ell+k}).$$
	Therefore, inequality \eqref{cota1-3} can be rewritten as $|e^{\Lambda_3}-1| < 8721506 \alpha^{-3n/2}$. Furthermore, if $n\geq 25$ then $|\Gamma_3|<0.002$. So, by applying Lemma \ref{lem:Weger}, we deduce that
	$$
	|\Lambda_3|<-\dfrac{\log(1-0.002)}{0.002} |\Gamma_3|< 8730240 \exp\left( -0.6 n\right).
	$$
	Put
	$$
	\vartheta_1:=-\log\alpha,\quad \vartheta_2:=\log 10,\quad \psi:= \log \left(\dfrac{d}{9c_{\alpha}^{k+\ell}}\right),\quad c:=8730240,\quad \delta:=0.6.$$
	Moreover, as $\max\{m,(k+\ell) n+k(k-1)/2+\ell(2k\ell-1)/2\}<5\times 10^{17}$, then we consider $X_0=5\times 10^{17}$.
	
	With the help of \texttt{Maple}, we realize that
	\[
	q_{43}=3468665590923027810230
	\]
	verifies  the conditions of Lemma \ref{lem5} for all $1\leq a\leq 9$, $1\leq k\leq 6$ and $1\leq\ell\leq 7$. Therefore, we apply Lemma \ref{lem5} to see that if Diophantine equation \eqref{main:eqn3} has solutions then 
	$$
	n\leq \dfrac{1}{0.6} \times \log\left(\dfrac{3468665590923027810230^2 \times 8730240}{\log 10 \times 5.4\times 10^{17}}\right)< 123.
	$$ 
	Now, for the second reduction of the bound of $n$,  we apply again Lemma \ref{lem5} by taking $X_0=1793$ and $q_{14}=9120227$, satisfying the conditions of Lemma \ref{lem5}. Thus, we obtain
	$$
	n\leq \dfrac{1}{0.6} \times \log\left(\dfrac{9120227^2 \times 8730240}{\log 10 \times 1793}\right)<67.
	$$
	Hence, the final step consists of checking equation \eqref{main:eqn3}, for $1\leq n \leq 68$, $1\leq k\leq 6$, $1\leq \ell \leq 7$, $2\leq m \leq 1009$ and $1\leq d \leq 9$. This is done by a simple routine written in Maple, which (in a few minutes) doesn't return any solution of the Diophantine equation \eqref{main:eqn3}. This completes the proof of Theorem \ref{main:thm3}.
\end{proof}

\vskip20pt\noindent {\bf Acknowledgements.} First auther thanks AMS-Simons Foundation for their generous funding. This paper was completed while the third author was visiting Xavier University of Louisiana. He thanks the institution for the great working environment, the hospitality and the support.


\begin{thebibliography}{1}\footnotesize
	
	\bibitem{BGL} E. F. Bravo, C. A. G\'omez and F. Luca, Product of consecutive Tribonacci numbers with only one distinct digit, {\it J. Integer Sequences} {\bf 22} (2019), Article 19.6.3.
	
	\bibitem{BMS} Y. Bugeaud, M. Mignotte and S. Siksek, Classical and modular approaches to exponential Diophantine equations I. Fibonacci and Lucas perfect powers, {\it Ann. of Math.} {\bf 163} (2006), 969--1018.
	
	\bibitem{FM} V. Fac\'{o} and D. Marques, Tribonacci Numbers and the Brocard-Ramanujan Equation, {\it Journal of Integer Sequences,} {\bf 19}
	(2016), Article 16.4.4.
	
	\bibitem{IT} N. Irmak and A. Togb\'e, On repdigits as product of
	consecutive Lucas numbers,  {\it Notes Number Theory Discrete Math.} {\bf 24} (2018), 95--102.
	
	\bibitem{FL} F. Luca, Fibonacci and Lucas numbers with only one distinct
	digit, {\it Portugal. Math.} {\bf 57} (2000), 243--254.
	
	\bibitem{DM} D. Marques, On $k$-generalized Fibonacci numbers with only
	one distinct digit, {\it Utilitas Math.} {\bf 98} (2015), 23--31.
	
	\bibitem{MT} D. Marques and A. Togb\'e, On repdigits as product of
	consecutive Fibonacci numbers, {\it Rend. Istit. Mat. Univ. Trieste.}
	{\bf 44} (2012), 393--397.
	
	\bibitem{M} E. M. Matveev, An explicit lower bound for a homogeneous
	rational linear form in the logarithms of algebraic numbers II,
	{\it Izv. Ross. Akad. Nauk Ser. Mat.} \textbf{64} (2000), 125--180.
	In Russian.  English translation in {\it Izv. Math.} \textbf{64}
	(2000), 1217--1269.
	
	\bibitem{RP} S. G. Rayaguru and G. K. Panda, Repdigits as products of consecutive Balancing or Lucas-Balancing numbers, {\it Fibonacci Quart.} {\bf 56.4} (2018),  319--324.
	
	\bibitem{RT} S. E. Rihane and A. Togb\'e, Padovan and Perrin numbers as product of two repdigits, {\it 
		Bolet\'in de la Sociedad Matema\'aica Mexicana}, 
	Volume 28,  {\bf 51}, (2022).
	
	\bibitem{slone} N. J. A. Sloane et al., {\em The On-Line Encyclopedia of
		Integer Sequences}, 2019.
	
	\bibitem{Weger:1989}
	B. M. M. de Weger, \emph{Algorithms for Diophantine equations}. PhD thesis, Eindhoven University of Technology, Eindhoven, the Netherlands, 1989.
	\bibitem{Y:2020}
	P. T. Young, 2-adic properties of generalized Fibonacci numbers, {\it Integers}  {\bf 20} (2020), \# A71.
	
\end{thebibliography}
\end{document}